\documentclass[10pt]{article}
\usepackage{amsmath, amsthm}
\usepackage{amssymb}
\usepackage[hypertex]{hyperref}

\usepackage{graphicx}

\newcommand{\hookuparrow}{\mathrel{\rotatebox[origin=c]{90}{$\hookrightarrow$}}}

\newtheorem{theorem}{Theorem}
\newtheorem{proposition}[theorem]{Proposition}
\newtheorem{lemma}[theorem]{Lemma}
\newtheorem{corollary}[theorem]{Corollary}

\newtheorem{definition}[theorem]{Definition}
\newtheorem{definition and remark}[theorem]{Definition and Remark}
\newtheorem{definition and proposition}[theorem]{Definition and Proposition} 
\newtheorem{remark and definition}[theorem]{Remark and Definition}
\newtheorem{example}[theorem]{Example}
\newtheorem{remark}[theorem]{Remark}

\newtheorem{notation and remark}[theorem]{Notation and Remark}

\newtheorem{fact}[theorem]{FACTS}

\newcommand {\PP}{\mathbb{P}}
\newcommand {\QQ}{\mathbb{Q}}

\newcommand{\End}{{\rm End}}

\newcommand{\Ann}{{\rm Ann}}

\renewcommand{\max}{{\rm Max}}

\newcommand{\ga}{\mathfrak{a}} 

\newcommand{\gm}{\mathfrak{m}}

\newcommand{\gp}{\mathfrak{p}}

\title{The quadratic complete intersections associated with the action of the symmetric group}
\author{T.\  Harima, Niigata University \\ Department of Mathematics Education, Niigata, 950-2181 Japan
\thanks{Supported by Grant-in-Aid for Scientific Research (C) (23540052).} 
\and A.\ Wachi, Hokkaido University of Education \\  Department of Mathematics,  
\\ Kushiro, 085-8580 Japan
\thanks{Supported by Grant-in-Aid for Scientific Research (C) (23540179).} 
\and  J.\ Watanabe, Tokai University \\ Department of Mathematics, Hiratsuka, 259-1292 Japan
\thanks{Supported by Grant-in-Aid for Scientific Research (C) (23540050).}}
\begin{document}

\maketitle
\date{}

\def\pa{{\partial}}
\begin{abstract} We prove that any quadratic complete intersection with a  certain action of  the symmetric group has the 
strong Lefschetz property over a field of characteristic zero. Furthermore we discuss under what condisions 
its ring of invariatns by a Young subgroup is a homogeneous complete intersection with a standard grading.     
\end{abstract}

\section{Introduction}
It seems natural to conjecture that all (Artinian) complete intersections with standard grading have the strong Lefschetz property 
over a field of characteristic zero. If there is a group action on a complete intersection, it sometimes 
enables us to prove that the ring  has the  property (see \cite{HMMNWW}\;Chapter~4). 
For example,  consider the monomial complete intersection:
\[A=K[x_1,x_2, \cdots, x_r]/(x_1^{n_1+1}, x_2^{n_2+1}, \cdots, x_r^{n_r+1}).\]
In spite of the simple nature of the assertion of the strong Lefschetz property, 
the proof it has the strong Lefschetz property is complicated.   
If $n_1=n_2=\cdots = n_r=1$, however,  Ikeda's Lemma provides an easy 
proof (\cite{ikeda}\; Lemma~1.1, \cite{HMMNWW}\; Proposition~3.67, Corollary~3.70).   
It seems remarkable that any monomial complete intersection appears as a subring of the quadratic monomial 
complete intersection.  In fact  the algebra $A$ above 
is the invariant subring of
the quadratic complete intersection  
$K[x_1,x_2, \cdots, x_n]/(x_1^2, x_2^2, \cdots, x_n^2)$ 
under the group action of the Young subgroup 
\[S_{n_1}\times S_{n_2} \times \cdots \times S_{n_r}  \subset S_n\]
where \[n=n_1 + n_2 + \cdots + n_r.\] 
Once we know the strong Lefschetz property in the quadratic case, the general case then follows almost 
immediately. 

The purpose of this paper is to  generalize this argument.  
First we construct a flat family of quadratic complete intersections,  with four parameters, 
on which the Young subgroup acts in the same way as it does on the quadratic  
monomial complete intersection. 
It will be proved that any member in this family  has the strong Lefscehtz property. 
It is crucial to assume that the generators of the defining ideal are quadrics and the symmetric group acts on it. 
Then we prove that the ring of invariants of any complete intersection  in this family 
by the action of any Young subgroup in $S_n$ is again  a complete intersection 
with  the strong Lefschetz property. 
Proof is easy, but the invariant subrings basically do not always have a standard grading.  
Rather surprisingly, however, it turns out that most of them have the standard grading 
thanks to the assumption that the generators are quadrics.  
  
The main results of this paper are Theorem~\ref{main_thm1} and Theorem~\ref{main_thm2} which are 
stated in \S4 and \S5 respectively.   
A theorem of Goto says that the ring of invariants of a complete intersection is again a complete intersection if the 
group is generated by pseudo-reflections and its order is invertible in the ground field. 
We need to construct a set of uniform generators for all
invariant subrings in the family. This is treated in the Appendix. 

The authors would like to thank Larry Smith very much for helpful comments to improve this paper. 
\section{Definitions}

\begin{definition}
{\rm 
Let $V = \bigoplus _{i=0} ^{\infty} V_i$ be a finite dimensional graded vector space and 
let $L \in \End _{{\bf gr}} (V)$ be a graded endomorphism 
\[L : V  \to V\]
of degree one. Namely a graded endomorphism of degree one is a collection of 
homomorphisms  $\{L_i:V_i \to V_{i+1}\}$. 
We call $L$ a {\bf weak Lefschetz element} if the map $L$ has  piece-wise full rank, i.e, the restricted map 
$L_i:V_i \to V_{i+1}$ is either injective or surjective for all $i = 0,1,2,\cdots$. 
We will write $L_i: V_i \to V_{i+1}$ simply as $L: V_i \to V_{i+1}$. 
We say that $L$ is a {\bf strong Lefschetz element} if there exists an integer $c$ such that 
$V_i=0$ for all $i \geq c+1$ and  the map $L^{c-2i}$ restricted to the homogeneous part 
$L^{c -2i} :V_i \to V_{c-i} $ is bijective for all $i=0,1,2, \cdots , [c/2].$   
The map $i \mapsto \dim _K V_i$ is called the Hilbert function of $V$.  Sometimes it is written as the 
power series $\sum _{i=0} ^{\infty} (\dim _K V_i)T^i$.  Since $V$ is a finite dimensional vector space, 
the Hilbert series of $V$ is actually  a polynomial in $T$.   
If a graded homomorphism of $L \in \End _{{\bf gr}}(V)$ is a strong 
Lefschetz element, it automatically implies that the  Hilbert 
function of $V$ is symmetric about the half integer $c/2$, where  $c=a+b$, $a$ is the initial and 
$b$ the end degrees of $V=\bigoplus _i V_i$. 
} 
\end{definition}

\begin{lemma}   \label{wl_implies_sl}
Suppose that $V=\bigoplus _i V_i$ is a finite dimensional graded vector space and that 
$V$ has a constant Hilbert function.  
Suppose that $L \in \End _{{\bf gr}}(V)$ is a graded endomorphism of degree one.  
If $L$ is a weak Lefschetz element, then $L$ is a strong Lefschetz element.  
\end{lemma}

\begin{proof} \label{weak_is_strong} 
Let $a$ be the initial and $b$ the end degrees of $V$, and let $c=a+b$. 
Then obviously the map 
$L^{c-2i}:V_i \to V_{c-i}$ is a bijection for any $i \leq [c/2]$.  
\end{proof}

\begin{definition}
{\rm  
Let $A= \bigoplus _{i=0} ^c A _i$ be a graded (not necessarily standard graded) Artinian $K$-algebra over 
a field with $A_0=K$.   
We say that $A$ has the {\bf weak} (resp. {\bf strong}) {\bf Lefschetz property},  if there exists 
a linear form  $l \in A_1$ such that the multiplication map 
$L=\times l \in \End _{{\bf gr}} (A)$ is a weak (resp. strong)  Lefschetz element.  
Such a linear form  $l$ is called a {\bf weak} (resp. {\bf strong}) {\bf Lefschetz element}. 
Sometimes we use the abbreviation: WLP (resp. SLP) for weak (resp. strong) Lefschetz property.
}
\end{definition} 

\begin{definition}  
{\rm  
Let $A= \bigoplus _{i=0} ^c A _i$ be a graded Artinian $K$-algebra over a field with $A_0=K$.  
The {\bf Sperner number} of $A$ is defined by  
\[\mbox{Sperner}\; A = \max\, _{i} \{\dim _K A_i\}.  \] 
}
\end{definition} 

\begin{proposition}[Subring Theorem]  \label{subring_theorem}
Let $A=\bigoplus _{i=0} ^c A_i$  be a graded Artinian $K$-algebra with the strong Lefschetz property.  
Assume that $A_c \not = 0$.  
Suppose that $B$ is a graded $K$-subalgebra of $A$, that $B_c=A_c$,  and $B_1$ contains a strong Lefschetz element for $A$. 
Then if $B$ has a symmetric Hilbert function, $B$ has the strong Lefschetz property.  
\end{proposition} 

\begin{proof}
Let $l \in B_1$ be a strong Lefschetz element for $A$.  Consider the diagram:
\[
\begin{array}{ccc}
A_i         & \to        & A_{c-i} \\ 
\hookuparrow    &     & \hspace{-3ex}\hookuparrow \\  
B_i         & \to        & B_{c-i},  
\end{array}
\] 
where the vertical arrows are natural injections and horizontal arrows are 
the multiplication map by $l^{c-2i}$.    
The strong Lefschetz property of $A$ implies that $\times l^{c-2i} : B_i \to B_{c-i}$ is injective. 
Since $\dim _K\;(B_i)=\dim _K\; (B_{c-i})$, it is bijective.  
\end{proof}

\newcommand {\slx}{{\mathfrak{sl} _2}}

\section{The polynomial ring and the action of the symmetric group}
Let $R=K[x_1, x_2, \cdots, x_n]$  be the polynomial ring over $K$,  a field of characteristic zero, and  
let $S_n$ be the symmetric group.  The homogeneous part of $R$ of degree $d$ is 
denoted by $R_d$.   We let the symmetric group $S_n$ act on $R$ by permutation of the variables. 
An element  $\sigma \in S_n$ is a bijection 
of the set $\{1,2, \cdots, n\}$.  Thus 
$\sigma$ induces the automorphism of the $K$-algebra $R$ by 
\[f^{\sigma}(x_1, x_2, \cdots, x_n)= f(x_{\sigma(1)}, x_{\sigma(2)}, \cdots, x_{\sigma(n)}).\] 
We recall some basic facts on the representation of $S_n$ and its  action  on $R$ and fix some notation. 
\begin{fact} \label{fact_on_representation}
\rm{
\begin{enumerate}
\item \label{fact_1} 
The irreducible representations of $S_n$ are parametrized by the Young diagrams of $n$ boxes.
A Young diagram of $n$ boxes is denoted by a partition $\lambda \vdash n$, which is 
a non-decreasing sequence of positive integers 
$\lambda =(\lambda _1, \lambda _2, \cdots, \lambda _k)$ such that $\sum \lambda _i=n$.   
\item \label{fact_2}
We will denote by $V^{\lambda}$ the irreducible module (uniquely determined up to isomorphism) 
corresponding to $\lambda$. 
The dimension of $V^{\lambda}$ is determined by the hook length formula. (See e.g., \cite{gJaK}, \cite{bS}).) 
Mostly we are interested in partitions of $n$ with at most two rows. Such partitions will be denoted as  
\[(n,0), (n-1,1), \cdots, (n-[n/2], [n/2]).\]
Note that  $(n,0)$ denotes the partition with one row.  
\item \label{fact_3}
Let $U$ be a finite dimensional $S_n$-module. The vector space $U$ decomposes as 
$U=\bigoplus _{\lambda \vdash n}U_{\lambda}$, where $U_{\lambda}$ is a sum of  copies of $V^{\lambda}$.  The number of times 
the irreducible module $V^{\lambda}$ occurs in $U$ is the multiplicity of $V^{\lambda}$.         
Such a decomposition of $U$ is unique up to order of the factors. 
In other words if  
$U= \bigoplus _{\lambda \vdash n} U_{\lambda} = \bigoplus _{\lambda \vdash n} U_{\lambda}'$, 
then 
$U_{\lambda}=U_{\lambda}'$ (as vector subspaces  of $U$) for all $\lambda \vdash n$. 
Such a decomposition is called the {\bf isotypic decomposition} of $U$. 
\item  \label{fact_4}
We denote by $Y^{\lambda}$ the Young symmetrizer corresponding to $\lambda \vdash n$. 
For the meaning of Young symmetrizers we refer the reader to \cite{gJaK} or \cite{bS}. 
In the sequel all we have to know about $Y^{\lambda}$ is that it gives  the 
projection on to the $\lambda$ isotypical summand 
\[Y^{\lambda}:U \to U_{\lambda},\]
for any  $S_n$-module $U$, where $U=\bigoplus _{\lambda} U_{\lambda}$ is the  isotypic decomposition. 
For example, if $\lambda =(n,0)$, then $Y^{\lambda}(R)$ coincides with  the ring  $R^{S_n}$ of invariants 
of $R$ under the action of $S_n$ and $Y^{\lambda}$ is the usual averaging homomorphism. 
It is well known that $R^{S_n}$ is, as a $K$-algebra, generated by the  elementary symmetric polynomials. 
The elementary symmetric polynomial of degree $d$ will be denoted by $e_d$.  Thus 
we have \[Y^{(n,0)}(R)=R^{S_n}=K[e_1, e_2, \cdots, e_n].\]
\item  \label{fact_5}
The degree one part  $R_1$ of $R$ decomposes, as an  $S_n$-module, as 
\[R_1 \cong V^{(n,0)} \oplus  V^{(n-1, 1)}.\]

Typical bases for these modules are: 
\begin{eqnarray*}
\langle x_1 + x_2 + \cdots + x_n   \rangle  &  \mbox{ for }   &    V^{(n,0)} \\
\langle x_1 -x_2, x_1 - x_3 ,  \cdots , x_1 - x_n   \rangle   & \mbox{ for } &    V^{(n-1,1)} 
\end{eqnarray*}

\item  \label{fact_6}
The degree two part  $R_2$ of $R$ decomposes, as an  $S_n$-module, as 
\[R_2 \cong V^{(n,0)} \oplus  V^{(n,0)} \oplus  V^{(n-1, 1)}  \oplus V^{(n-1, 1)}  \oplus V^{(n-2, 2)}.\]

For $V^{(n,0)}$ we can choose $\langle e_1^2 \rangle$ and $\langle e_2 \rangle$ as bases. 

For $V^{(n-1,1)}$ we can choose 
$\langle x^2_1 -x^2_2, x^2_1 - x^2_3 ,  \cdots , x^2_1 - x^2_n   \rangle$   and 
\newline 
$\langle (x_1 -x_2)e_1, (x_1 - x_3)e_1 ,  \cdots , (x_1 - x_n)e_1   \rangle$ as bases. 

A typical basis for $V^{(n-2,2)}$ is the set of Specht polynomials of shape $(n-2, 2)$: 
\[ \{(x_1-x_j)(x_2-x_k) \mid 3 \leq j < k \leq n \} \cup \{
(x_1-x_2)(x_3-x_k) \mid 4 \leq k \leq n \} \]
For the definition of Specht polynomials see \cite{HMMNWW}\; \S9.3.

\item  \label{fact_7}
By the hook length formula we have  \[\dim V^{(n-i,i)} = {n \choose i} -{n \choose i-1}.\]
In particular,  
\[\dim V^{(n-2,2)} = {n \choose 2} -{n \choose 1} = \frac{n(n-3)}{2},\]
and 
\[\dim R_2 = \left\{\begin{array}{l}
2 \dim V^{(n,0)} + 2 \dim V^{(n-1,1)} +  \dim V^{(n-2,2)}, \mbox{ if $n> 3$}, \\ 
2 \dim V^{(n,0)} + 2 \dim V^{(n-1,1)},  \mbox{ if $n =  3$}, \\ 
2 \dim V^{(n,0)} +  \dim V^{(n-1,1)}, \mbox{ if $n=2$}.
\end{array}\right. \] 
\end{enumerate}
} 
\end{fact}

\begin{lemma}   \label{one_dim_submodule_in_R_2} 
With the same notation as above, suppose that $U \subset R_1e_1$ is a one dimensional 
$S_n$ module. If $n \geq 3$, then 
$U$ is spanned by $e_1^2$.  If $n=2$, then $U$ is spanned either by $e_1^2=e_1(x_1+x_2)$ or $e_1(x_1-x_2)$. 
\end{lemma}

\begin{proof} Note that $U \cong R_1$  as an $S_n$-module.  
Since $R_1=Y^{(n,0)}(R_1) \oplus Y^{(n-1,1)}(R_1)$ is the isotypic decomposition, 
and $\dim Y^{(n-1,1)}(R_1) =1$ if $n=2$ and $ \dim Y^{(n-1, 1)}(R_1) > 1$ if  $n \geq 3$.  
Thus the assertion follows. 
\end{proof} 

\section{Main result} \label{main_result} 

As in the previous section  $R=K[x_1, x_2, \cdots, x_n]$ denotes  the polynomial ring over a field  $K$ of 
characteristic zero. 
We are interested in the sequences of  quadrics  $f_i \in R$   which satisfy the following conditions: 
\begin{enumerate}
\item[(1)] 
$f_1, f_2, \cdots, f_n$ form a regular sequence. 
\item[(2)]
For any $\sigma \in S_n$, 
\[f_i(x_{\sigma(1)}, x_{\sigma(2)}\cdots, x_{\sigma(n)}) = f_{\sigma(i)}(x_1,x_2, \cdots, x_n), 
\mbox{ for }i=1,2, \cdots, n.\]
\end{enumerate}

\begin{remark}   \label{notation_for_generators}
{\rm  If $(f_1, f_2, \cdots, f_n)$ is a sequence of quadrics in $R$ with the second property above, 
the stabilizer of $f_1$ must be the subgroup 
$S_{n-1}$ of $S_n$ which fixes 1. Hence the element  $f_1$  has the from 
\[p_0x_1 ^2 + p_1(x_2+ x_3 + \cdots + x_n )x_1 + p_2(x_2^2+x_3^2+ \cdots + x_n^2) + p_3(\sum _{2 \leq i < j \leq n}x_ix_j),\]
with four parameters $p_k$ and the elements $f_i$ are obtained by cyclically permuting the variables.
For such a sequence to be a regular sequence it is necessary and sufficient that the resultant does not vanish.
We refer the interested reader to~\cite{GKZ} for details on the resultants of homogeneous forms.   
}
\end{remark}

\begin{example}\rm{
Put $R=K[x,y,z]$, $e=x+y+z$, $f=(e-ax)(e-bx)$, $g=(e-ay)(e-by)$, $h=(e-az)(e-bz)$. 
Assume that \[ab(a-3)(b-3)(ab-a-2b)(ab-2a-b) \neq 0.\]
(This is equivalent to the resultant of $f,g,h$ mod the exponents.) 
Then the sequence \[f,g,h\] satisfy the conditions (1) and (2) of the first paragraph of this section.   
}
\end{example}

\begin{theorem} \label{main_thm1}  Assume that the characteristic of  $K$ is zero.  
Let $I=(f_1, f_2, \cdots, f_n)$ be a complete intersection ideal in $R$ which satisfies 
the conditions (1) and (2) above. 
Then $A:=R/I$ has the strong Lefschetz property. 
Let $e_1=\sum x_i$. If $e_1^2 \not \in I$, then 
$e_1$ is a strong Lefschetz element for $A$. 
\end{theorem}

\begin{proof}
If $e_1^2 \in I$, we can choose $e_1^2$ as a generator of the ideal $I$.  Let $B=K[z]/(z^2)$,  
with a new variable $z$,  
and define the map $B \to A$ by $z \mapsto e_1$.  
It is easy to see that $B \to A$ is a flat extension and the fiber, say $C$, is the algebra 
\[C=K[x_1, x_2, \cdots, x_n]/(e_1,f_1, f_2, \cdots, f_{n}).\]
For $i \geq 2$, let $f_i'$ be the polynomial obtained from $f_i$ by the substitution 
\[x_1 \mapsto -(x_2+x_3 + \cdots + x_n).\]
It is easy to see that  
\[C \cong K[x_2, x_3, \cdots, x_n]/(f_2', f_3', \dots, f_n')\] 
so if $\sigma \in S_{n}$ fixes $1$,  then  $(f_i')^{\sigma}=f'_{\sigma(i)}$. 
Hence we may induct on $n$ to conclude that the fiber has the SLP. 
By the Flat Extension Theorem (\cite{HMMNWW} Theorem~4.10), the ring 
$A$ has the strong Lefschetz property.  

For the rest of proof we assume that $I$ does not contain $e_1^2$. 
We want to show that $(I \cap e_1R) \cap R_2 =0$ if $n \geq 3$.
Let  $ h \in (I \cap e_1R) \cap R_2$.  
Since $I$ is generated by a regular sequence, any two linearly independent elements in 
$I \cap R_2$ are not  contained in a principal ideal. 
This implies that  $(I \cap e_1R) \cap R_2$ is at most one dimensional. 
If $\sigma \in S_n$, it forces $h^{\sigma}R = hR$.   In other words $h$ is a semi-invariant.  
By Lemma~\ref{one_dim_submodule_in_R_2}, the element $h$ is a scalar multiple of $e_1^2$ if $n \geq 3$. 
Since we have assumed that $I$ does not contain $e_1^2$, we have $h=0$. 
If $n=2$, $h$ could be $x_1^2 - x_2^2$, but in any case  $A$ has the SLP as  is easily checked.  

From now on we assume that $n \geq 3$.  
Then the sum $R_1e_1 + (I \cap R_2)$ is a direct sum and it contains two copies of $V^{(n,0)}$ 
and two copies of $V^{(n-1,1)}$, since both of $R_1e_1$ and  $(I \cap R_2)$ are equivalent to $V^{(n,0)} \oplus V^{(n-1,1)}$.  
  On the other hand by FACTS~\ref{fact_on_representation}\;(\ref{fact_5} and \ref{fact_6}), 
$R_1e_1 + (x_1^2 , \cdots, x_n^2)$  also contains two copies of  $V^{(n,0)}$ and two  copies of $V^{(n-1,1)}$.  
By FACTS~\ref{fact_on_representation}\;(\ref{fact_3}), we see that 
\[R_1e_1 + (I \cap R_2) = R_1e_1 + (x_1^2 , \cdots, x_n^2) \cap R_2,\]
 and  
\[I+e_1R=(x_1^2, x_2^2, \cdots, x_n^2, e_1).\] 
In particular the ideal 
$I + e_1R$ contains all the second power of the variables. 

Put $B=R/(x_1^2, x_2^2, \cdots, x_n^2)$.  
Generally, it is the case 
that  $\dim _K\; (A/e_1A) \geq \mbox{Sperner}\; (A)$. (See the proof of \cite{HMMNWW}~Proposition~3.5.)  
Thus we have 
\[\dim _K (B/e_1B) = \dim _K (A/e_1A) \geq  \mbox{Sperner}\; (A) = \mbox{Sperner}\; (B), \]
and since $B$ has the weak Lefschetz property (\cite{HMMNWW} Corollary~3.69), 
this implies that \linebreak[4] $\dim _K (A/e_1A) = \mbox{Sperner}\; (A)$. 
Hence $A$ has the weak Lefschetz property.

We have to prove that $A$ has the strong Lefschetz property with $e_1$ as an SL element. 
Let $J=(x_1^2, x_2^2, \cdots, x_n^2)$.  
Put $\lambda _i=(n-i, i)$ for $i=0,1,2,\cdots, [n/2]$. Since the 
way $S_n$ acts on $R/J$ and $R/I$ are the same, 
$A$ and $B$ are isomorphic as $S_n$-modules. 
So we may apply \cite{HMMNWW} Theorem~9.9 to $A$. 
Since the $S_n$-module  $A$ decomposes as 
\[A=\bigoplus _{i=0}^{[n/2]}Y^{\lambda _i}(A)\] and since the multiplication map 
$\times e_1:A \to A$ decomposes as the sum of the restricted maps 
\[\times e_1: Y^{\lambda _i}(A)  \to Y^{\lambda _i}(A),\ i=0,1, \cdots, [n/2],\] 
it suffices to prove that the endomorphism  
$\times e_1: Y^{\lambda _i}(A)  \to Y^{\lambda _i}(A)$ is a strong Lefschetz element 
for each $i$. 
Recall that $Y^{\lambda _i}(A)$ has 
a constant Hilbert function (\cite{HMMNWW} Lemma~9.8), namely its 
Hilbert function is 
\[\left( \dim V^{(n-i,i)} \right)\left(   T^i+T^{i+1}+ \cdots + T^{n-i} \right) .\]
Thus $A$ has the  SLP by Lemma~\ref{wl_implies_sl}. 
\end{proof}

\section{Some consequences}
Recall that a grading of a graded algebra $A=\bigoplus _{i\geq 0} A_i$ is 
{\bf standard} if the algebra $A$ is generated by  elements of degree one over $A_0$.   
So far we have tacitly assumed that the grading for the algebras $R$ and $A$ are standard.  
In this section we consider graded subalgebras which are not necessarily standard graded.  
We continue to assume that the polynomial ring $R$ has the standard grading, i.e., the 
degrees of the variables  are one, but the invariant subrings $R^G$ and $A^G$  most likely do not 
have the standard grading.  We are primarily concerned however with the cases where the  
invariant subrings  $A^G$ {\em do} have the standard grading.    
\begin{theorem} \label{main_thm2} 
Let $A=K[x_1, x_2, \cdots, x_n]/(f_1, f_2, \cdots, f_n)$ be a quadratic complete intersection 
with the action of $S_n$ as in Theorem~\ref{main_thm1}.
As in Remark~\ref{notation_for_generators} we use the notation  
 \[f_1=p_0x_1 ^2 + p_1(x_2+ x_3 + \cdots + x_n )x_1 + p_2(x_2^2+x_3^2+ \cdots + x_n^2) + p_3(\sum _{2 \leq i < j \leq n}x_ix_j).\]
Let $X=\{x_1, x_2, \cdots, x_n\}$ be the set of variables and 
let $X=\bigsqcup _{i=1}^r X_i$ be a partition of the set of variables into $r$ nonempty subsets.   
Put $n_i=\mid X_i\mid$ and let 
\[G=S_{n_1} \times S_{n_2} \times \cdots  \times S_{n_r} \]  be the Young subgroup of 
$S_n$ which acts on $R$ in such a way that $S_{n_k}$ permutes the variables in the block $X_k$ and leaves fixed the 
variables in other blocks.  Assume that $e_1^2 \not \in I$. 
Then  $R^G/(I \cap R^G)$ is a complete intersection with the strong Lefschetz property. 
Let  $S=K[y_1, y_2, \cdots, y_r]$ be the polynomial ring in $r$ variables and let 
\[\phi :S \to A\] 
be the homomorphism defined by $\phi(y_i)= \sum _{x \in X_i}x$. 
Then the image  $\phi(S)$ coincides with $A^G$ for any $(p_0, p_1, p_2, p_3)$ 
in a nonempty open set in the projective space $\PP ^3=\{(p_0, p_1, p_2, p_3)\}$. 
In particular $A^G$ has the standard grading if parameters are general enough.    
\end{theorem}

\begin{proof}
The Young subgroup $G \subset S_n$ is generated by reflections. By a theorem of Goto \cite{Goto_1}, the ring of invariants 
$A^G=(R/I)^G=R^G/(I \cap R^G)$ is a complete intersection. Note that $e_1 \in A^G$ and $e_1$ is an SL element for $A$. 
Moreover the image of the Jacobian determinant $\left|\frac{\pa f_i}{\pa x_j}\right|$ in $A$  
can be be taken as a socle generator for both $A$ and $A^G$.  
By Proposition~\ref{subring_theorem} the ring $A^G$ has 
the strong Lefschetz property if $e_1 \in A$ is a strong Lefschetz element. 
This is a proof for the first assertion of this theorem. 
For the second assertion we prove Lemma~\ref{torsion_part} first.

\begin{lemma} \label{torsion_part} 
Let $Q$ be a Noetherian integral domain. Suppose that $R=\bigoplus _{i\geq 0} R_i$ is a graded $Q$-algebra,  
finitely generated over $R_0=Q$. (We assume that the graded pieces $R_i$ are free $Q$-modules.) 
If for some maximal ideal $\gm_0 \subset Q$, the fiber  $R/{\gm _0}R:=R\otimes _Q Q/{\gm _0}$ has a standard grading, 
then there exists an ideal  $\ga \neq 0$ in $Q$ such that 
$R/{\gp}R$ has the standard grading for all prime ideals $\gp  \not  \supset \ga$. 
\end{lemma}

\begin{proof}
Let $Y=\{Y_1, Y_2, \cdots, Y_r\}$  be a set of homogeneous elements in $R$ such that 
$Y$ generates the algebra: $R=Q[Y_1, Y_2, \cdots, Y_r]$. Let $M', M''$ be the $Q$-submodules of 
$R$ as follows: 
${\displaystyle M'=\sum _{i_1+i_2+\cdots + i_r \geq 1}QY_1^{i_1}Y_2^{i_2}\cdots Y_r^{i_r},}$
and 
${\displaystyle M''=\sum _{i_1+i_2+\cdots + i_r \geq 2}QY_1^{i_1}Y_2^{i_2}\cdots Y_r^{i_r}  }$. 
Furthermore put
$M=M'/M''$.  
Note that $M$ is a graded $Q$-module and $M \otimes _Q Q/\gm$ is the tangent space of $R \otimes _Q Q_{\gm}$ 
for any maximal ideal $\gm  \subset Q$.   
It is easy to see that the fiber $R \otimes _Q Q/\gm $ 
has the standard grading if and only if  $M \otimes _Q Q/\gm$  is spanned by 
the homogeneous elements of degree one. 
Let $N$ be the $Q$-submodule of $M$  generated by 
the degree one elements and put $\ga = \Ann _Q (M/N)$. 
Recall that $\gp \supset \ga \Leftrightarrow (M/N)_{\gp} \neq 0$ for a prime ideal $\gp$ in $Q$.
(See \cite{matsumura}~The paragraph preceding Theorem~4.4.)  
Since there exists at least one maximal ideal such that 
$M \otimes _Q Q/\gm$ has the standard grading, we have $\ga \neq 0$.  
Then it is straightforward that $\ga$ has the desired property. 
\end{proof}

{\it Proof of the second part of Theorem~\ref{main_thm2}.}
Define the polynomials  $F$ and  $F_i$  by 
\[F=P_0x_1^2 + P_1(\sum _{j=2}^n x_j)x_1+ P_2(\sum _{j=2}^nx_j^2) + P_3(\sum _{2 \leq k < l \leq n}x_kx_l), \]
and 
$F_i=F^{\sigma ^{i-1}}, i=1,2, \cdots, n$, where  $P_0, \cdots, P_3$ are indeterminates and 
$\sigma=(12\cdots n)$ is the cycle of length $n$. 
These  are considered as  polynomials in the variables $x_1, x_2, \cdots, x_n$
with coefficients in  $K[P_0, P_1, P_2,P_3]$.    
Let $\cal{R}$ be the resultant of $F_1, F_2, \cdots, F_n$. 
Put $Q=K[P_0,P_1, P_2, P_3, {\cal R}^{-1}]$ and we consider the algebra 
$Q[x_1, \cdots, x_n]/(F_1, \cdots, F_n)$. 
In the decomposition of the variables   $X=\bigsqcup _{i=1}^r X_i$ we may assume 
that the blocks $X_i$ consist of variables of consecutive indices.  
So we assume that the $i$-th block $X_i$ is  
\[ X_i=\{x_{n_1+ \cdots + n_{i-1}+1}, \cdots, x_{n_1+\cdots + {n_{i}}}\}. \]
The numbering of the variables may be illustrated as follows:
\[\underbrace{x_1, \cdots, x_{n_1}}_{n_1}, \underbrace{x_{n_1+1}, \cdots, x_{n_1+n_2}}_{n_2}, x_{n_1+n_2+1},  
\cdots, \underbrace{x_{n-n_r + 1} \cdots, x_{n}}_{n_r}.\]
For the sake of notation we rename the variables as : 
\[x_{ij}=  x_{n_1+ n_2 + \cdots + n_{i-1} +j},\]
so $x_{ij}$ is the $j$-th variable in the $i$-th block.  
Suppose that $X_0$ is a set of  variables.  
Then by $e_d(X_0)$ we denoted the elementary symmetric polynomial of degree $d$ in the variables in $X_0$. 

Introduce a set of new variables $Y_{ij}$ of degree one which are indexed as follows:
\begin{enumerate}
\item
The first index $i$ ranges $i=1,2, \cdots, r$.
\item
The second index $j$ ranges, depending on $i$, over  $j=1,2, \cdots, n_i$.
\end{enumerate}
Define the polynomials $\{ E_{ij}\}$  with the same indices as the variables $\{Y_{ij}\}$
as follows: 
\[ E_{id}=e_{d}(\{Y_{i1}, Y_{i2}, \cdots, Y_{in_i}\}) \mbox{ for } i=1,2, \cdots, r,  \;  d= 1, 2, \cdots, n_i. \]
Define the algebras $\Lambda$ and $\Lambda '$  by  
\[\Lambda = Q[\{Y_{ij}\}]/(F_1, \cdots, F_n),\] 
\[\Lambda '=Q[\{E_{ij}\}]/((F_1, \cdots, F_n) \cap Q[\{E _{ij}\}].\]
Note that $\Lambda$ is mapped onto $A$ by the specialization
$P_k \mapsto p_k, Y_{ij} \mapsto x_{ij}$. 
By 
Corollary~\ref{uniform_generator} in the Appendix, 
it is possible to write 
\[
(F_1, \cdots, F_n) \cap Q[\{E _{ij}\}]=(F_1', F_2', \cdots, F_n').
\]
(Note that $F_1', \cdots, F_n'$ are constructed from $F_1, \cdots, F_n$ explicitly.)  
The algebra $\Lambda '$ is a flat extension of $Q=K[P_0, P_1, P_2,P _3, {\cal R}^{-1}]$ and each fiber 
coincides  with $A^G$ under the map $Y_{ij} \mapsto x_{ij}$.  

On the other hand the image of $\phi: S \to A$ is a subring of $A^G$ and they coincide if and only if 
$A^G$ has the standard grading or equivalently $A^G$ is generated by degree one elements.  
For $P_0=1, P_1=P_2=P_3=0$, it is easy to see that the fiber has the standard grading (cf.\ \cite{HMMNWW}\;Lemma~3.70).  
Thus we may apply Lemma~\ref{torsion_part}. This completes the proof of Theorem~\ref{main_thm2}.  
\end{proof}

\begin{remark} \label{monomial_subring} 
\rm{
Let $\Lambda$ be the algebra defined in the last paragraph of the proof of the second part of Theorem~\ref{main_thm2}. 
The fiber for $p_0=1, p_1=p_2=p_3=0$ is isomorphic to  
the  monomial complete intersection
\[K[y_1, y_2, \cdots, y_r]/(y_1^{n_1 +1}, y_2^{n_2 +1}, \cdots,  y_r^{n_r +1}).\]
This is proved  if  $r=2$ in   \cite{HMMNWW}\;Lemma~3.70. The same proof in fact  works for all $r$.   
}
\end{remark}

The following example shows that a member can fail to have the SLP in a flat family of Artinian algebras 
whose general members have the SLP.  It also shows that the embedding dimension is not a constant in a flat family 
of Artinian algebras. 
\begin{example} 
\rm{
Let $p, x, y$ be  variables and consider $K[p,x,y]/(y^2, x^3-py)$. 
We regard it as a family of Artinian algebras. 
Give the variables $p,x,y$ degrees $0,1,3$ respectively.
For any  $p \in K$,  the Hilbert function of $R$ is 
\[\frac{(1-T^3)(1-T^6)}{(1-T)(1-T^3)}=\frac{(1-T^3)}{(1-T)}(1+T^3)=\frac{1-T^6}{1-T}=1+T+ \cdots + T^5.\]
If $p=0$, the fiber is $K[x,y]/(x^3, y^2)$ and if $p \neq 0$, then 
it is $K[x]/(x^6)$.  
}
\end{example}

The following example illustrates Theorem~\ref{main_thm2}.  
\begin{example} {\rm 
Consider the family of the Artinian algebras  
\[K[p_0, \cdots, p_3][v,w,x,y,z]/(f_1, f_2, f_3, f_4, f_5),\] 
on which $S_5$ acts by the permutation of the variables $\{v, w, x, y, z\}$    
and 
$f_i^{\sigma} =f_{\sigma (i)}$ for  $\sigma \in S_5$. 
We will use the same notation as the first paragraph of 
Section~\ref{main_result}, and 
Remark~\ref{notation_for_generators} with $n=5$, so  
\[f_1=p_0v^2 + p_1(w+x+y+z)v+ p_2(w^2+x^2+y^2+z^2) + p_3(wx+wy+wz+xy+xz+yz).\]
Other generators $f_2, \cdots, f_5$ are obtained by permuting the variables.  
Consider the Young subgroup 
\[G:=S_2 \times S_3,\]
which acts on $A$ with the division of the variables:
\[\{v,w,x,y,z\}=\{v,w\} \sqcup \{x,y,z\}.\] 

Then the ring $A^G$ of invariants 
has the Hilbert function 
\[(1 \ 2 \ 3 \ 3 \ 2 \ 1 ),\]
and most cases it is generated by degree one elements.   
But for $p_0=5, p_1=2, p_2=0, p_3=2$, the algebra $K[(A^G)_1] \subset A^G$ has the Hilbert function
\[(1 \ 2 \ 2 \ 2 \ 2 \ 1 ).\]
Some more such examples are: 
\[ (p_0, p_1, p_2, p_3)=(0, 0, 3, 8) \]
\[ (p_0, p_1, p_2, p_3)=(7, 7, 3, 8) \]
\[ (p_0, p_1, p_2, p_3)=(4, 3, 2, 6) \]
\[ (p_0, p_1, p_2, p_3)=(6, 0, 0, 4) \]
\[ (p_0, p_1, p_2, p_3)=(6, 3, 0, 2) \]
\[ (p_0, p_1, p_2, p_3)=(1, 1, 3, 8) \]
}
\end{example}

If $A=R/I$ is not a quadratic complete intersection, then in the general case  the ring of invariants of  
$A$ does not have the standard grading. In the next example we exhibit such a case.

\begin{example}
{\rm 
Let $R=\QQ[x_1, x_2, \cdots,  x_6]$, $I=(f_1, \cdots , f_6)$, where 
$f_i= x_i^3$ for $i=1, \cdots, 6$. 
Let $X_1=\{x_1, x_2, x_3\}$ and $X_2=\{x_4, x_5, x_6\}$ and 
let $G=S_3 \times S_3$ act on $A=R/I$ by permuting the variables within the blocks $X_1$ and $X_2$ 
in the way as described in Theorem~\ref{main_thm2}. 
Then the ring of invariants $A^G$ is, as a $K$-algebra, generated by the six elements of degrees $\{1,2,3,1,2,3\}$ as follows:
$r=x_1+x_2+x_3$, $s=x_1x_2+x_1x_3+x_2x_3$, $t=x_1x_2x_3$, $u=x_4+x_5+x_6$, $v=x_4x_5+x_4x_6+x_5x_6$ and $w=x_4x_5x_6$. 
They satisfy the relations 
\begin{enumerate}
\item $u^3-3uv+3w=0$, 
\item $r^3-3rs+3t=0$, 
\item $u^2v-2v^2-uw=0$, 
\item $r^2s-2s^2-rt=0$, 
\item $u^2w - 2vw=0$, 
\item $r^2t-2st=0$.  
\end{enumerate} 
Note that the ring  $A^G$ has in fact embedding dimension 4 and $t$ and $w$ can be eliminated 
but the grading is not standard. 
The Hilbert polynomial is 
\[1+2T+5T^2+8T^3+12T^4+14T^5+16T^6+14T^7+12T^8+8T^9+5T^{10}+2T^{11}+T^{12}\]
\[=((1+T^2)(1+T+T^2+T^3+T^4))^2.\]
}
\end{example}

\section{Appendix}
\begin{proposition} \label{generators_of_differential_module}
Let $R=K[x_1, \cdots, x_n]$ be the polynomial ring over a field 
$K$ of characteristic zero, on which the symmetric group $S_n$ acts  by 
permuting the variables. 
Let $f_1, \cdots, f_n \in R$  be a set of homogeneous elements which satisfies 
$f_i^{\sigma} = f_{\sigma(i)}$ for any  $\sigma \in S_n$ for $i=1, \cdots , n$.  
Define the polynomials $g_1, g_2, \cdots, g_n$ by 
\[\begin{pmatrix} g_1 \\ g_2 \\ \vdots \\ g_n \end{pmatrix}=
\begin{pmatrix} 
1&1& \cdots & 1 \\
x_1&x_2& \cdots & x_n \\
x_1^2&x_2^2& \cdots & x_n^2 \\
x_1^{n-1} & x_2^{n-1} & \cdots & x_n ^{n-1} 
\end{pmatrix}
\begin{pmatrix} f_1 \\ f_2 \\ \vdots \\ f_n \end{pmatrix}.
\]
Then the ideal $(f_1, \cdots, f_n)$ is a complete intersection if and only if 
$(g_1, \cdots, g_n)$ is a complete intersection.  
\end{proposition}

\begin{proof}
The ``if'' part is obvious. 
Put  $I=(f_1, \cdots, f_n)$ and assume that $I$ is  an ideal of finite colength, i.e., $I$ is a complete intersection. 
We want to prove that $g_1, \cdots, g_n$ generate an ideal of finite colength. 
For simplicity we put $G=S_n$. 
It is possible to construct a minimal free resolution of $(f_1, \cdots, f_n)$ which is 
compatible with the action of $G$.  For this the Koszul complex is enough:
\[0 \to \bigwedge ^n F      \to \cdots \to \bigwedge ^2 F   \to \bigwedge ^1 F \to \bigwedge ^0 F. \]
We may think $F$ is the free module generated by  $dx_1,dx_2, \cdots , dx_n$, and then  
 extend the action of $G$ to $\bigwedge ^k F$ (for any $k$) in the obvious manner. 
If $k=1$, it is easy to determine a minimal set of  generators for $F^G$ as an $R^G$-module. 
This can be done as follows.  First we note that $F ^G$ is a free $R^G$-module of rank $n$. 
For any $d$, the ideal 
$(x_1^d, x_2^d, \cdots, x_n^d) \cap R^G$ is generated by 
the power sum symmetric polynomials of degrees $d, d+1, \cdots, d+n-1$. (This is discussed in the proof of \cite{HW_21}~Lemma~7.6.)   
Hence it follows that the  invariant subspace $F^G$ of 
$F$ is a free $R^G$-module of rank $n$ generated by 
$\{\sum _{i=1}^nx_i^kdx_i\mid k=0,1,\cdots, n-1\}$.  In other words a matrix 
${\bf M}$ is determined to be the Van der Monde matrix if it satisfies,  for any $d$ given,  the following matrix identity:
\[\begin{pmatrix}  x_1^d+x_2^d+ \cdots + x_n^d \\ 
x_1^{d+1}+x_2^{d+1}+ \cdots + x_n^{d+1} \\ 
 \vdots  \\ 
x_1^{d+n-1}+x_2^{d+n-1}+ \cdots + x_n^{d+n-1} \\ 
\end{pmatrix}={\bf M}
\begin{pmatrix} x_1^d \\ x_2^d \\ \vdots \\ x_n^d \end{pmatrix}.
\]
The first part of the minimal free resolution of  $R/I$ takes the form 
\[\bigwedge ^2 F  \to  \bigwedge ^1 F  \stackrel{\partial}{\to} R \to R/I \to 0.\]
To extract the invariant subspace is an exact functor, so we have the exact sequence 
\[(\bigwedge ^2 F)^G  \to  F^G  \stackrel{\partial}{\to} R^G \to (R/I)^G \to 0\]
of free $R^G$-modules. The map $\partial : F \to R$ is defined by 
\[dx_i \mapsto f_i.\]
Thus the image of the restricted map $\partial :F ^G \to R^G$ is the ideal 
\[(g_1, g_2, \cdots, g_n)R^G.\] 
This shows that $I \cap R^G$ is an ideal of finite colength in $R^G$ or equivalently 
they generate an ideal of finite colength in $R$.  
\end{proof}
We will call the matrix in the statement of Proposition~\ref{generators_of_differential_module}
the Van der Monde matrix.
\begin{corollary} \label{uniform_generator}
Let $n=n_1 + \cdots + n_r$ be a partition of the integer $n$ and 
let  
\[\underbrace{x_1, \cdots, x_{n_1}}_{n_1}, \underbrace{x_{n_1+1}, \cdots, x_{n_1+n_2}}_{n_2}, x_{n_1+n_2+1},  
\cdots, \underbrace{x_{n-n_r + 1} \cdots, x_{n}}_{n_r}\]
be a decomposition of the variables into $r$ blocks. 
Let $G=S_{n_1}\times S_{n_2} \times \cdots \times S_{n_r} \subset S_n$ be a Young subgroup and 
let $G$ act on $R$ be the block-wise permutation of the variables. 
Suppose that $f_1, \cdots, f_n$ is a homogeneous complete intersection
which satisfies 
$f_i^{\sigma} = f_{\sigma (i)},\; i= 1, 2, \cdots, n,$ for all $\sigma \in G.$ 
Let $V_i$ be the Van der Monde matrix in the variables in the $i$-th block 
\[\{x_{n_1+ \cdots + n_{i-1} +j} \; \mid \; j=1 ,\cdots, n_{i}\}.\] 
Define the homogeneous elements $g_1, g_2, \cdots, g_n$ 
by 
\[\begin{pmatrix} g_1 \\ g_2 \\ \vdots \\ g_n \end{pmatrix}=
\begin{pmatrix} 
V_1& 0 & \cdots   & 0 \\  
0 & V_2 & \cdots  & 0 \\  
\vdots & \vdots &  \ddots  & \vdots  \\
0   & 0 & \cdots & V_r \\  
\end{pmatrix}
\begin{pmatrix} f_1 \\ f_2 \\ \vdots \\ f_n \end{pmatrix}.
\]
(The matrix is a block diagonal matrix.) 
Then $(g_1, \cdots, g_n)=(f_1, f_2, \cdots, f_n) \cap R^G$. 
\end{corollary}
\begin{proof}
$(g_1, \cdots, g_n) \subset (f_1, f_2, \cdots, f_n) \cap R^G$
is obvious. 
Put   
\[\Omega = Rdx_1 \oplus Rdx_2 \oplus \cdots \oplus Rdx_n,\; I=(f_1, f_2, \cdots, f_n).\]
Construct a minimal free resolution of $R/I$ over $R$ as:
\[0 \to \bigwedge ^n\Omega  \to \cdots \to \bigwedge ^1 \Omega  \stackrel{\partial}{\to} R \to R/I \to 0,\]
so that  the boundary maps are compatible with the action of the group $G$. 
By taking the invariant subspaces for $G$ we may get the minimal free resolution 
of $(R/I)^G$. As in Proposition~\ref{generators_of_differential_module}, 
 the invariant subspace $\Omega ^G$ as an $R^G$-module can be generated by the elements which appear as 
the entries of the column vector: 
\[
\begin{pmatrix} 
V_1    & 0      & \cdots   & 0 \\  
0      & V_2    & \cdots   & 0 \\  
\vdots & \vdots & \ddots  & \vdots  \\
0      & 0      & \cdots  & V_r   
\end{pmatrix}
\begin{pmatrix} 
dx_1 \\ dx_2 \\ \vdots \\ dx_n 
\end{pmatrix}.
\]
The map $\Omega \stackrel{\partial}{\to} R$  is defined as 
$dx_i \mapsto f_i$. Hence 
we obtain the module $R^G/(R^G \cap (f_1, \cdots, f_n))$ as 
represented by $R^G/(g_1, \cdots, g_n)R^G$.   
\end{proof}

\begin{remark}
We have been unable to determine a set of generators for $(\bigwedge ^k F)^G$ for $k >1$ except 
for $k=n-1, n$.   
If $k=n$, then $F^n$ is the free $R^G$-module of rank one with 
\[\left(\prod _{1 \leq k < l \leq n}(x_l-x_k)\right) dx_1\wedge dx_2 \wedge \cdots \wedge dx_n\]
as a generator. 
If $k=n-1$, a set of generators can be specified similarly. 
\end{remark}


\begin{thebibliography}{99}

\bibitem{Goto_1} 
S.\ Goto, 
\emph{Invariant subrings under the action by a finite group generated by pseudo-reflections}, 
Osaka J.\ Math.\ , \textbf{15}, no.\ 1,  47--50 (1978) 
\bibitem{HW_21} 
T. Harima and J. Watanabe, 
\emph{The strong Lefschetz property for Artinian algebras with non-standard grading}, 
J.\ Algebra \textbf{311} (2), 511--537 (2007)  
\bibitem{HMMNWW}
T.\ Harima, T.\ Maeno, H.\ Morita, Y.\ Numata, A.\ Wachi,  and   J.\ Watanabe  
The Lefschetz Properties, 
Springer Lecture Notes~2080, Springer-Verlag, 2013.
\bibitem{ikeda}
H.\ Ikeda,  
\emph{Results on  Dilworth and Ress Numbers of Artinian local rings}, 
Japan J.\ Math.\ (N.S.)  \textbf{22}(1), 147--158 (1996)  
\bibitem{GKZ}
I.\ M.\ Gelfand, M.\ M.\ Kapranov, A.\ V.\  Zelevinsky, 
Discriminants, Resultants and Multidimensional Determinants,  
Birkh\"{a}user,  1994. 
\bibitem{gJaK}
G.\ James and A.\ Kerber, The representation theory of the symmetric group, 
Encyclopedia of Mathematics and its applications, 16, Addison-Weslay, Reading, 1981.
\bibitem{matsumura}
H.\ Matsumura,  Commutative Algebra Theory, 
Cambridge Studies in Advaced Mathematics, 8, Cambridge University Press, 1989.  
\bibitem{bS}
B.\ Sagan, The symmetric group, Second Edition, Graduate Text in Mathematics, 203, Springer-Verlag, New York, 2001. 
\end{thebibliography}
\end{document}